\documentclass[12pt,twoside, reqno]{amsart}
\usepackage[pagewise]{lineno}
\usepackage[all]{xy}   
\usepackage {amssymb,latexsym,amsthm,amsmath}
\usepackage[colorlinks, linkcolor=blue, citecolor=red]{hyperref}
\usepackage[utf8]{inputenc}
\usepackage{enumitem} 
\usepackage{tikz}
\usepackage{float}
\usetikzlibrary{arrows.meta}
\usepackage{xcolor}
\usepackage{todonotes}
\usepackage{graphicx}

\long\def\symbolfootnote[#1]#2{\begingroup%
\def\thefootnote{\fnsymbol{footnote}}\footnote[#1]{#2}\endgroup}
\newcommand{\tr}{\ensuremath{{}^t\!}}

\newcommand{\Aut}{\textup{Aut}}

\def \diag {\mathrm{diag}}

\def \GL {\mathrm{GL}}
\def \SL {\mathrm{SL}}
\def \U {\mathrm{U}}
\def \SU {\mathrm{SU}}
\def \Sp {\mathrm{Sp}}

\def \Int {\mathrm{Int}}
\def \G {\mathrm{G}}

\newtheorem{theorem}{Theorem}[section]
\newtheorem{lemma}[theorem]{Lemma}
\newtheorem{corollary}[theorem]{Corollary}
\newtheorem{proposition}[theorem]{Proposition}
\newtheorem*{theorem*}{Theorem}
\theoremstyle{definition}
\newtheorem{remark}[theorem]{Remark}
\newtheorem{definition}[theorem]{Definition}

\numberwithin{equation}{section}

\newcommand{\ignore}[1]{}

\newcommand{\mynote}[1]{}
\newcommand{\secref}[1]{Section~\ref{#1}}

\newcommand{\lemref}[1]{Lemma~\ref{#1}}

\newcommand{\corref}[1]{Corollary~\ref{#1}}

\begin{document}

\setcounter{section}{0}
\setcounter{tocdepth}{1}
\title[Chirality and Non-real elements in $\G_2(q)$]{Chirality and Non-real elements in $\G_2(q)$}
\author[Sushil Bhunia]{Sushil Bhunia}
\author[Amit Kulshrestha ]{Amit Kulshrestha }
\author[Anupam Singh]{Anupam Singh}
\address{Department of Mathematics, BITS-Pilani, Hyderabad Campus, Hyderabad, India}
\email{sushilbhunia@gmail.com; sushil.bhunia@hyderabad.bits-pilani.ac.in}
\address{IISER Mohali, Knowledge City, Sector 81, Mohali 140 306, India}
\email{amitk@iisermohali.ac.in}
\address{IISER Pune, Dr. Homi Bhabha Road, Pashan, Pune 411 008, India}
\email{anupamk18@gmail.com}
\thanks{The first-named author is funded by SERB grant MTR/2023/001279 and BITS Pilani NFSG grant NFSG/HYD/2023/H0853 for this research. 
The second-named author thanks 
DST-FIST facility established through grant SR/FST/MS-I/2019/46 to support this research.
The third-named author is funded by an NBHM research grant 02011/23/2023/NBHM(RP)/RDII/5955 for this research.}
\dedicatory{In memory of Professor Nikolai Vavilov}
\today
\begin{abstract}
In this article, we determine the non-real elements --- the ones that are not conjugate to their inverses --- in the group $G=\G_2(q)$ when $char(\mathbb F_q)\neq 2, 3$. We use this to show that this group is chiral; that is, there is a word $w$ such that $w(G)\neq w(G)^{-1}$. We also show that most classical finite simple groups are achiral.  
\end{abstract}
\maketitle 

\section{Introduction}
Let $w$ be a word in the free group $F_d$. For any group $G$, we have the following map called \emph{word map} $w\colon G^d\rightarrow G $
given by evaluation $(g_1, g_2, \ldots, g_d) \mapsto w(g_1, g_2, \ldots, g_d)$. The image $w(G^d)$ of $w$ is simply denoted by $w(G)$.
Some of the general questions in the subject are about image sizes, fibers, and Waring-like problems. Several fundamental results have been proved for finite simple and quasisimple groups (see, for example,~\cite{LST11, GLOST18}). In~\cite{Lu14}, Lubotzky characterized that exactly those subsets of a finite simple group that are invariant under automorphisms of the group and contain identity are images of some word. For nilpotent groups, this need not be true, which is studied in~\cite{KKK22}. In this article, we explore another notion based on properties of images of word maps, namely chirality, for some finite groups of Lie type. 

A group $G$ is said to be \emph{chiral} if there is some word $w$ such that $w(G)\neq w(G)^{-1}$. Otherwise, $G$ is said to be \emph{achiral}. A word $w$ for which $w(G)\neq w(G)^{-1}$ is said to be a \emph{witness of chirality}. All Abelian groups are achiral since word maps (equivalently, power maps) are homomorphisms. In~\cite{ch19}, it is shown that the Mathieu group $M_{11}$ is chiral, and a witness to chirality is obtained. In~\cite{ch18}, the chirality of certain nilpotent groups is explored. In this article, we aim to study the property of chirality for the alternating groups, finite classical groups, and some exceptional groups. This problem is closely related to the problem of reality (see Lemma~\ref{techlemma}) and the understanding of automorphisms for these groups. In Section~\ref{classical}, we see that most classical groups are achiral. 

However, when we look at exceptional groups, it turns out $\G_2(q)$, when $char(\mathbb F_q)\neq 2, 3$, is chiral for certain $q$. More precisely (see Corollary~\ref{chiral-g2}) the group $\G_2(q)$ is chiral if and only if $q$ satisfies one of the following conditions:
\begin{enumerate}
\item the field $\mathbb F_q$ has a primitive $3^{rd}$ root of unity and doesn't have a primitive $9^{th}$ root of unity, i.e., $q \equiv 1 \mod 3$ and $q \not\equiv 1 \mod 9$.
\item $\mathbb F_{q^2}^1$ (the group of elements with norm $1$) contains a primitive $3^{rd}$ root of unity and $\mathbb F_{q^2}$ doesn't have $9^{th}$ primitive root of unity, i.e., $q \equiv 2 \mod 3$ and $q^2 \not\equiv 1 \mod 9$.
\end{enumerate}
Due to Lemma~\ref{techlemma}, the search for chirality requires identifying non-real elements in this group which are not fixed by field automorphisms. It turns out that there is a pair of non-real conjugacy classes, that is, there are two non-real conjugacy classes when they exist (depending on $q$). In Section~\ref{G2q}, first, we determine the non-real conjugacy classes which are of independent interest (see Theorem~\ref{maintheorem}). The study of real conjugacy classes done in~\cite{ST05, ST08} reduces our task to finding isolated matrices (see Definitions~\ref{slisolated} and~\ref{suisolated}) in $\SL_3(q)$ and $\SU_3(q)$. Then, we use this to establish that $\G_2(q)$ is chiral (Corollary~\ref{chiral-g2}). We hope to explore other exceptional groups in the future for chirality, too.   

\subsection*{Acknowledgement} We thank Professor Kunyavskii for his wonderful series of lectures in an online ``NCM workshop on Finite groups of Lie type" held in 2021, which motivated us to take up this work. 

\section{Preliminaries}

We study the property of chirality for some finite simple groups of Lie type, which is closely related to the problem of reality. An element $g\in G$ is \emph{real} if $g^{-1}=xgx^{-1}$ for some $x\in G$, and $G$ is called \emph{real} if every element of $G$ is real. 
First, we recall a few basic results to build criteria to test chirality. We begin with recalling~\cite[Lemma 2.1]{ch18}. 
\begin{lemma}\label{techlemma}
Let $G$ be a group such that for every $g\in G$ there is a homomorphism $\varphi\colon G\rightarrow G$, that may depend on $g$, such that $\varphi(g)=g^{-1}$. Then $G$ is achiral.
\end{lemma}
The following is an immediate consequence.
\begin{corollary}\label{realch}
If a group $G$ is real, then $G$ is achiral.
\end{corollary}
\begin{proof}
Since $G$ is real, for every $g$ in $G$ there is some $x \in G$ such that $g^{-1}=xgx^{-1}$. Choose $\varphi$ to be the inner automorphism $\Int_x$ and apply \lemref{techlemma}.
\end{proof}
\noindent In this article, we will provide several examples of achiral groups which are not real groups (for example, see \secref{classical}). Note that even when a group $G$ is not real, as long as its elements are mapped to their inverses by a homomorphism, it is achiral. The following basic lemma will be useful in our further explorations. 
\begin{lemma}\label{ses}
Let $1\longrightarrow N \overset{i}{\longrightarrow} G\overset{\pi}{\longrightarrow} Q\longrightarrow 1$ be a short exact sequence of groups. Then the following hold:
\begin{enumerate}[leftmargin=*]
\item\label{sesgq} If $G$ is achiral, then so is $Q$.
\item\label{sesgn} If $G$ is real then $N$ is achiral.
\end{enumerate}
\end{lemma}
\begin{proof}
\begin{enumerate}[leftmargin=*]
\item This follows from \cite[Proposition 4.1]{ch18}.
\item Let $n\in N$ then $n^{-1}=gng^{-1}$ for some $g\in G$ as $G$ is real. Therefore the map $\varphi:N\rightarrow N$ given by $\varphi(x)=gxg^{-1}$ (for all $x\in N$) is an automorphism of $N$ (as $N\lhd G$) such that $\varphi(n)=n^{-1}$. Thus by \lemref{techlemma}, $N$ is achiral.
\end{enumerate}
\end{proof}

\section{Chirality of classical groups}\label{classical}
In this section, we explore various classical groups and show several examples of achiral groups. We note that the automorphisms of groups of Lie type are well known due to the work of Dieudonné~\cite{Di51} and Steinberg~\cite{St60}. We use these automorphisms to show that elements of the group under consideration can be conjugated to its inverse by an automorphism (not necessarily an inner automorphism). 

\begin{proposition}\label{gln}
\begin{enumerate}[leftmargin=*]
\item The alternating group $\mathcal{A}_n$ is achiral.
\item The projective general linear group $\mathrm{PGL}_n(k)$ and its covering groups are achiral. 
\item  If $G$ is a group of type $A_{n-1}$, then $G$ is achiral.
\item The orthogonal group $\mathrm{O}_n(k)$ is achiral. Thus, simple groups of Lie type $B_l$ and $D_l$ are achiral.
\end{enumerate}
\end{proposition}
\begin{proof}
\begin{enumerate}[leftmargin=*]
\item  Since $\mathcal{A}_n\lhd S_n$ and $S_n$ is a real group then, in view of  \lemref{ses}\eqref{sesgn}, $\mathcal{A}_n$ is achiral.
\item Let $g\in \GL_n(k)$ then ${}^tg=xgx^{-1}$ for some $x\in \GL_n(k)$. Let $\gamma$ be the automorphism of $\GL_n(k)$ given by $\gamma(y)={}^ty^{-1}$ for all $y\in \GL_n(k)$. 
Then \[\varphi(g)= \Int_{x^{-1}}\circ\gamma(g)=\Int_{x^{-1}}({^t}g^{-1})=(x^{-1}{}^tgx)^{-1}=g^{-1}.\] Thus, by \lemref{techlemma} $\GL_n(k)$ is achiral. Note that in view of \lemref{ses}\eqref{sesgq} $\GL_n(k)/C$  is also achiral, for any central subgroup $C$ of $\mathcal Z(\GL_n(k))$.
\item Let $g\in \SL_n(k)$ then ${}^tg=xgx^{-1}$ for some $x\in \GL_n(k)$. Thus the map 
\[\varphi_{x^{-1}} \colon \SL_n(k)\rightarrow \SL_n(k)\] given by $\varphi_{x^{-1}}(h)=x^{-1}hx$ (for all $h\in \SL_n(k)$) is an automorphism of $\SL_n(k)$. Now, consider the automorphism $\gamma$ of $\SL_n(k)$ given by $\gamma(h)={}^th^{-1}$ for all $h\in \SL_n(k)$. Therefore the automorphism $\psi:=\varphi_{x^{-1}}\circ\gamma$ of $\SL_n(k)$ satisfies the following:
\begin{align*}
\psi(g)=\varphi_{x^{-1}}\circ\gamma(g)=\varphi_{x^{-1}}({}^tg^{-1})=x^{-1}{}^tg^{-1}x=g^{-1}.
\end{align*} 
Hence by \lemref{techlemma} $\SL_n(k)$ is achiral. Which in turn implies that $\SL_n(k)/C$ is also achiral by \lemref{ses}\eqref{sesgq} for any central subgroup $C$ of $\mu_n=\mathcal Z(\SL_n(k))$. In particular, $\mathrm{PSL}_n(k)$ is achiral.
\item Since $\mathrm{O}_n(k)$ is real then $\mathrm{SO}_n(k)$ and its derived subgroup $\Omega_n(k)$ is achiral (by \lemref{ses}\eqref{sesgn}). Therefore by \lemref{ses}\eqref{sesgq}, $\mathrm{P\Omega_n}(k)$ is also achiral.
\end{enumerate}
\end{proof}
We note that by~\cite[Corollary C]{FZ82} $\mathrm{Spin}_{2l+1}(q)$ is achiral. 

\begin{remark}
We don't have to assume that $G$ is finite in the above result.
\end{remark}
Let $\mathbb{F}_q$ be a finite field of odd characteristic. Consider the
unitary group 
$$\mathrm{U}_n(q):=\{g\in \GL_n(q^2)\mid {}^t\overline{g}g=I_n\},$$
and the symplectic group
\[\Sp_{2n}(q):=\{g\in \GL_{2n}(q)\mid {}^tgJg=J\}, \text{ where } J=\begin{pmatrix}0&I_n\\-I_n&0\end{pmatrix}.\] 
The extended symplectic group is $\Sp_{2n}^{\pm}(q):=\{g\in \GL_{2n}(q)\mid {}^tgJg=\pm J\}.$

\begin{proposition}\label{achiral-u}
\begin{enumerate}[leftmargin=*]
\item The unitary group $\mathrm{U}_n(q)$ is achiral.
\item If $G$ is a finite group of type ${}^2A_{n-1}$, then $G$ is achiral.
\end{enumerate}
\end{proposition}
\begin{proof}
\begin{enumerate}[leftmargin=*]
\item Let $g\in \mathrm{U}_n(q)$. Two elements of $\mathrm{U}_n(q)$ are conjugate in $\mathrm{U}_n(q)$ if and only if they are conjugate in $\GL_n(q^2)$ (by \cite[(6.1)]{Ma81}). Since ${}^tg\in \mathrm{U}_n(q)$, and $g$ and ${}^tg$ are conjugate in $\GL_n(q^2)$, we have that ${}^tg=xgx^{-1}$ for some $x\in \mathrm{U}_n(q)$. Thus the automorphism $\psi:=\Int_{x^{-1}}\circ\gamma$ of $\mathrm{U}_n(q)$ satisfies $\psi(g)=g^{-1}$, where $\gamma(y)={}^ty^{-1}$ for all $y\in \mathrm{U}_n(q)$. Hence, by \lemref{techlemma} $\mathrm{U}_n(q)$ is achiral.
\item For any $g\in \mathrm{SU}_n(q)$, the automorphism $\psi:=\varphi_{x^{-1}}\circ\gamma$ of $\mathrm{SU}_n(q)$ satisfies $\psi(g)=g^{-1}$ by similar computation. Thus by \lemref{techlemma} $\mathrm{SU}_n(q)$ is also achiral. This completes the proof by \lemref{ses}\eqref{sesgq}.
\end{enumerate}
\end{proof} 

\begin{proposition}\label{achiral-sp}
 If $G$ is a finite group of Lie type of $C_n$, then $G$ is achiral.
\end{proposition}
\begin{proof}
First note that if $k$ is a finite field of characteristic $2$, then $\Sp_{2n}(k)$ is real, and hence achiral by \corref{realch}.
If $q\equiv 1\;(\mathrm{mod}\;4)$, then $\Sp_{2n}(q)$ is real (see \cite[Theorem E (i)]{FZ82}), which in turn implies that $\Sp_{2n}(q)$ is achiral by \corref{realch}. Now if $q\equiv 3\;(\mathrm{mod}\;4)$, then by \cite[Theorem E(ii)]{FZ82} every element is conjugate to its inverse by an element of $\Sp_{2n}^{\pm}(q)$. Since $\Sp_{2n}(q)$ is a normal subgroup of $\Sp_{2n}^{\pm}(q)$, then for every element $g\in\Sp_{2n}(q)$ there is an automorphism $\varphi$ (say) of $\Sp_{2n}(q)$ such that $\varphi(g)=g^{-1}$. Hence, by \lemref{techlemma} $\Sp_{2n}(q)$ is achiral in this case, too. Therefore by \lemref{ses}\eqref{sesgq} $\mathrm{PSp}_{2n}(q)$ is also  achiral. 
\end{proof}

\section{Isolated matrices in $\SL_3(q)$}

To study chirality of $\G_2(q)$, we would require to understanding certain types of matrices in $\SL_3(q)$ and $\SU_3(q)$. We do this in this section and the following one. We begin with a definition. 
\begin{definition}\label{slisolated}
A matrix $A \in \SL_3(q)$ will be called \emph{isolated} if $A$ is neither conjugate to $\tr A$,  nor to $A^{-1}$. 
\end{definition}
\noindent Note that being isolated is a property of conjugacy class. The following lemma is an immediate consequence of the definition.
\begin{lemma}
If $A \in \SL_3(q)$ is isolated, then so are $\tr A, A^{-1}$ and $\tr A^{-1}$.  
\end{lemma}
\noindent Before proceeding further, we note the following. To understand the conjugacy class of $A\in \SL_n(q)$, one often examines if its conjugacy class in $\GL_n(q)$ is a union of more than one conjugacy class of $\SL_n(q)$.  Representatives of the classes in the union are conjugates of $A$ by an element of $\GL_n(q)$. This phenomenon is referred to as the \textit{splitting of a conjugacy class}. This is useful in identifying the cases when a matrix is not isolated. This is because $A$ and $\tr A$ are conjugates in $\GL_n(q)$; and if the conjugacy class of $A$ in $\GL_n(q)$ does not split in $\SL_n(q)$ then $A$ would not be isolated. In the case of $\SU_n(q)$, a similar strategy will be used in the next section. 

We recall the Lemma 6.6 and Theorem 6.8 from~\cite{ST05}.
\begin{proposition}\label{prop-norm-condition}
\begin{enumerate}[leftmargin=*]
\item[$(1)$] Let $A\in \SL_n(q)$ and 
$\mathcal Z_{\GL_n(q)}(A)$ be the centralizer of $A$ in $\GL_n(q)$. Let
$N_A = \det(\mathcal Z_{\GL_n(q)}(A))\subset \mathbb F_q^*$.
Then, $g_1Ag_1^{-1}$ is conjugate to $g_2Ag_2^{-1}$ in $\SL_n(q)$, where $g_1, g_2 \in \GL_n(q)$, if and only if $\det(g_1)\equiv \det(g_2) \mod N_A$. 

\item[$(2)$] Let $A\in \SU_n(q)$ and 
$\mathcal Z_{\U_n(q)}(A)$ be the centralizer of $A$ in $\U_n(q)$.
Let $N_A = \det(\mathcal Z_{\U_n(q)}(A)) \subset \mathbb F_{q^2}^1$.
Then, $g_1Ag_1^{-1}$ is conjugate to $g_2Ag_2^{-1}$ in $\SU_n(q)$, where $g_1, g_2 \in \U_n(q)$, if and only if $\det(g_1)\equiv \det(g_2) \mod N_A$.
\end{enumerate}
\end{proposition}

We use this to establish that most of the matrices in $\SL_3(q)$ are not isolated.   
\begin{lemma}\label{isolated-l1}
Let $A$ be in $\SL_3(q)$. Suppose, its characteristic polynomial $\chi_A(X)$ and
the minimal polynomial $m_A(X)$ do not satisfy $\chi_A(X) = (X - \alpha)^3 = m_A(X)$, where $\alpha^3 = 1$ for some $\alpha \in \mathbb F_q^*$. Then the conjugacy class of $A$ in $\GL_3(q)$ does not split in $\SL_3(q)$. Further, in these cases, $A$ is not isolated. 
\end{lemma}
\begin{proof}
Using Proposition~\ref{prop-norm-condition} we rule out the possibility of $A$ in $\SL_3(q)$ being isolated in all cases except when the characteristic polynomial is equal to the minimal polynomial given by $\chi_A(X) = (X - \alpha)^3 = m_A(X)$, where $\alpha^3 = 1$. We note that $A$ and $\tr A$ are always conjugates in $\GL_3(q)$. Thus, if the conjugacy class of $A$ in $\GL_3(q)$ does not split in $\SL_3(q)$, the matrix $A$ can not be isolated in $\SL_3(q)$. For this we need to compute $N_A=\det(\mathcal Z_{\GL_3(q)}(A))$ and show $N_A = \mathbb F_q^*$, hence the matrix $A$ can not be isolated. If $\chi_A(X)$ is a degree $3$ irreducible polynomial then the $\mathcal 
Z_{\GL_3(q)}(A) = \mathbb F_q[A]^\times \cong \mathbb F_{q^3}^*$. Since the norm map is surjective for finite extensions of finite fields, we are done in this case. When $\chi_A(X) = (X-a)p(X)$ where $p(X)$ is a degree $2$ irreducible then the $\mathcal Z_{\GL_3(q)}(A) = \mathbb F_q[A]^\times \cong \mathbb F_q^* \times \mathbb F_{q^2}^*$ and again $N_A$ is the whole of the set. When $\chi_A(X)$ has at least $2$ distinct roots, we can easily compute the centralizer and show $N_A = \mathbb F_q^*$. This leaves out the case when $\chi_A(X) = (X-\alpha)^3$. The case that $m_A(X)= (X- \alpha)$ or $(X-\alpha)^2$ can be dealt with similarly as the centralizer would be $\GL_3(q)$ in the first case and would have a block of $\GL_2(q)$ in the second case. This completes the proof.
\end{proof}

Thus, to find an isolated matrix $A$ we need to look into the case $m_A(X) = (X-\alpha)^3$ with $\alpha^3 = 1$. 
\begin{lemma}\label{isolated-l2}
Let $A\in \SL_3(q)$ with the minimal polynomial $m_A(X) = (X-\alpha)^3$. Then, the conjugacy class of $A$ in $\GL_3(q)$ splits in $\SL_3(q)$ if and only if $q\equiv 1 \mod 3$. Hence, if $q\not\equiv 1 \mod 3$, the matrix $A$ is not isolated. 
\end{lemma}
\begin{proof}
We note that in this case $\chi_A(X) = m_A(X)=(X-\alpha)^3$ the matrix $A$ is conjugate to $\begin{pmatrix} \alpha &1& \\ & \alpha &1 \\ &&\alpha \end{pmatrix}$ in $\GL_3(q)$. Further, its centralizer in $\GL_3(q)$ is given by 
$$\mathcal Z_{\GL_3(q)}(A)= \left\{\begin{pmatrix} a &b&c \\ & a 
&b \\ &&a \end{pmatrix}\mid a\in \mathbb F_q^*, b,c\in \mathbb F_q\right\}.$$ 
Hence, in this case $N_A= (\mathbb F_q^*)^3 \subset \mathbb F_q^* $. Thus, the conjugacy class of $A$ in $\GL_3(q)$ splits into 
$|\frac{\mathbb F_q^* }{(\mathbb F_q^*)^3}| = 1$ or $3$ classes in $\SL_3(q)$. 

Let $\phi \colon \mathbb F_q^* \rightarrow \mathbb F_q^*$ be the group homomorphism given by $x\mapsto x^3$. Note that the conjugacy class of $A$ in $\GL_3(q)$ splits if and only if $|\frac{\mathbb F_q^* }{(\mathbb F_q^*)^3}| \neq 1$, which is equivalent to $\ker(\phi) \neq 1$. This, in turn, is equivalent to $3 \mid (q-1)$. This proves the required result. 
\end{proof}

Let us fix $\alpha \in \mathbb F_q^*$ with $\alpha^3=1$. Such a non-trivial $\alpha$ exists in $\mathbb F_q $ if and only if $q\equiv 1 \mod 3$. This is also a condition required for the splitting of the conjugacy class of $A$. For the matrix $A=\begin{pmatrix} \alpha &1& \\ & \alpha &1 \\ &&\alpha \end{pmatrix}$ the conjugacy class of $A$ in $\GL_3(q)$ splits into the following three classes in $\SL_3(q)$: 
$A$, $A_{1} = \diag(\alpha, 1, 1).A. \diag(\alpha, 1, 1)^{-1} = \begin{pmatrix} \alpha &\alpha& \\ & \alpha &1 \\ &&\alpha \end{pmatrix}$ and $A_{2}=\diag(\alpha^2, 1, 1).A. \diag(\alpha^2, 
1, 1)^{-1} = \begin{pmatrix} \alpha &\alpha^2 & \\ & \alpha &1 \\ &&\alpha \end{pmatrix}$. 

\begin{lemma}\label{isolated-l3}
The matrix $A$ is conjugate to $\tr A$ and hence is not an isolated matrix.
\end{lemma}
\begin{proof} We check that $\begin{pmatrix} & &-1 \\ &-1&\\ -1&&  \end{pmatrix} A=\tr A  \begin{pmatrix} & &-1 \\ &-1&\\ -1&&  \end{pmatrix}$ in $\SL_3(q)$.
\end{proof}
\begin{lemma}\label{isolated-l4}
The matrix $A_1$ (respectively $A_2$) is conjugate to $\tr A_1$ (respectively $\tr A_2$) if and only if $X^3=\alpha^{-1}$ (respectively $X^3=\alpha^{-2}$) has a solution.  
\end{lemma}
\begin{proof} Let $B=(b_{ij})\in \SL_3(q)$ be such that $BA_1=\tr A_1B$. A simple computation leads to $B=\begin{pmatrix} && b_{13} \\ &b_{13}\alpha & b_{23}\\ b_{13} & b_{23} & b_{33} \end{pmatrix}$ 
with the condition $\det(B)= -b_{13}^3\alpha = 1$. The required condition is equivalent to $X^3=\alpha^{-1}$.  A similar calculation can be done for $A_2$.
\end{proof}

\begin{theorem}\label{isolated-t5}
Let $char(\mathbb F_q)\neq 2, 3$ and $A\in \SL_3(q)$. Then, the isolated matrices in $\SL_3(q)$ exist if and only if (the field $\mathbb F_q$ has a primitive $3^{rd}$ root of unity and doesn't have a primitive $9^{th}$ root of unity) $q \equiv 1 \mod 3$ and $q \not\equiv 1 \mod 9$. Further, there are exactly two conjugacy classes given by the matrices $A_{1} = \begin{pmatrix} \alpha &\alpha& \\ & \alpha &1 \\ &&\alpha \end{pmatrix}$ and $A_{2}= \begin{pmatrix} \alpha &\alpha^2 & \\ & \alpha &1 \\ &&\alpha \end{pmatrix}$ which are isolated. Here, $\alpha$ is a primitive $3^{rd}$ root of unity. Further, it turns out that $A_1$ is conjugate to $\tr A_2$. 
\end{theorem}
\begin{proof}
From Lemmas~\ref{isolated-l1}, \ref{isolated-l2}, \ref{isolated-l3}, it follows that the isolated matrices can be found if and only if $q \equiv 1 \mod 3$. Further, the only possibilities for isolated matrices, up to conjugacy, are $A_1$ and $A_2$ (and their inverses). From Lemma~\ref{isolated-l4}, it follows that $A_1$ and $A_2$ are not conjugate to their transpose if and only if $\mathbb F_q$ does not have a primitive $9^{th}$ root of unity, which happens if and only if  $q \not\equiv 1 \mod 9$. Further, $A_1$ and $A_2$ both are not real because of the eigenvalue considerations.

One can check that $\begin{pmatrix} && -1 \\ &-\alpha^2 & \\ -\alpha && \end{pmatrix} A_1 = \tr A_2 \begin{pmatrix} && -1 \\ &-\alpha^2 & \\ -\alpha && \end{pmatrix} $.
\end{proof}

\section{Isolated matrices in $\SU_3(q)$}

We note that all hermitian forms over a finite field are equivalent, thus giving conjugate groups as unitary groups. We choose one which makes the calculations convenient, namely the following: ${\mathcal B}= \begin{pmatrix} &&1\\ &1& \\ 1&&\end{pmatrix}$ and 
$$\U_n(q)=\{A\in \GL_n(q^2) \mid \tr{A} {\mathcal B} \bar A={\mathcal B}\}.$$
Consider the norm map $N\colon \mathbb F_{q^2} \rightarrow \mathbb F_{q^2}$ which is $x \mapsto x\bar x=x^{1+q}$. Denote the norm $1$ elements by $\mathbb F_{q^2}^1$, which is of size $q+1$ and gives the determinant of $\U_n(q)$. Similar to the case of $\SL_3(q)$, we define the notion of isolated matrices here too, which is a property of the conjugacy class.
\begin{definition}\label{suisolated}
A matrix $A \in \SU_3(q)$ is called \emph{isolated} if $A$ is neither conjugate to ${\bar A}^{-1}$ nor to $A^{-1}$.
\end{definition}
\noindent The following is immediate:
\begin{lemma}
If $A \in \SU_3(q)$ is isolated, then so are $\bar A$, ${\bar A}^{-1}$, and $A^{-1}$.
\end{lemma}

To find isolated matrices, we take a similar approach as in the case of $\SL_3(q)$. We know that for $A\in \SU_3(q)$ we have $A$ is conjugate to $\tr{\bar A^{-1}}$ in $\U_3(q)$, which is conjugate to $\bar A^{-1}$. This is because conjugacy classes of $\GL_n(q^2)$ are in one-to-one correspondence with that of $\U_n(q)$. Thus, by analyzing how conjugacy classes of $A$ taken in $\U_n(q)$ split in $\SU_n(q)$, we will narrow down our problem to a few cases. For this, we use Proposition~\ref{prop-norm-condition}.  

\begin{lemma}\label{isolated-lu1}
Let $A$ be in $\SU_3(q)$ of which the characteristic polynomial and the minimal polynomial do not satisfy $\chi_A(X) = (X - \alpha)^3 = m_A(X)$ where $\alpha^3 = 1$ and $\alpha\in \mathbb F_{q^2}^1$. Then the conjugacy class of $A$ in $\U_3(q)$ does not split in $\SU_3(q)$. Further, in these cases, $A$ is not isolated. 
\end{lemma}
\begin{proof}
In view of Proposition~\ref{prop-norm-condition} we need to rule out all other cases of $A$ in $\SU_3(q)$ except when the characteristic polynomial is equal to the minimal polynomial given by $\chi_A(X) = (X - \alpha)^3 = m_A(X)$ where $\alpha^3 = 1$ and $\alpha\in \mathbb F_{q^2}^1$. For this, we need to show that in all those cases, $N_A=\det (\mathcal Z_{\U_3(q)}(A))$ is the whole of $\mathbb F_{q^2}^1$. The centralizer of various representatives of conjugacy classes in $\SU_3(q)$ is available in \cite[Section 4]{SS20}. Using that, we get our desired result.
\end{proof}

\begin{lemma}\label{isolated-lu2}
Let $A\in \SU_3(q)$ with the minimal polynomial $m_A(X) = (X -\alpha)^3$ for some $\alpha\in \mathbb F_{q^2}^1$. Then, the conjugacy class of $A$ in $\U_3(q)$ splits in $\SU_3(q)$ if and 
only if ($\mathbb F_{q^2}^1$ contains a primitive $3^{rd}$ root of unity) $q\equiv 2 \mod 3$. Hence, if $q\not\equiv 2 \mod 3$, the matrix $A$ is not isolated. 
\end{lemma}
\begin{proof}
First we find an appropriate representative $A$ with $\chi_A(X) = m_A(X) = (X-\alpha)^3$ in $\SU_3(q)$. Fix $\alpha \in \mathbb F_{q^2}^1$ with $\alpha^3=1$. We can take $A=\begin{pmatrix} \alpha & -\alpha^2 & -\frac{\alpha}{1+\alpha^2} \\ & \alpha & 1 \\ &&\alpha \end{pmatrix}$ and verify that $\tr{A} {\mathcal B} \bar A={\mathcal B}$. Now the centralizer of $A$ in $\U_3(q)$ will have the form $\left\{\begin{pmatrix} a &*&* \\ & a &* \\ &&a \end{pmatrix}\mid a\in 
\mathbb F_{q^2}^1\right\}$ and hence $N_A= (\mathbb F_{q^2}^1)^3$. Thus, the conjugacy class of $A$ in $\U_3(q)$ splits in $|\frac{\mathbb F_{q^2}^1 }{(\mathbb F_{q^2}^1)^3}| = 1$ or $3$ classes 
in $\SU_3(q)$. 

Let $\phi \colon \mathbb F_{q^2}^1 \rightarrow \mathbb F_{q^2}^1$ be the group homomorphism given by $x\mapsto x^3$. Note that the conjugacy class of $A$ in $\U_3(q)$ splits if and only if $ |\frac{\mathbb F_{q^2}^1 }{(\mathbb F_{q^2}^1)^3}| \neq 1$ if and only if $ker(\phi) \neq 1$ if and only if $3 \mid (q+1)$. This proves the required result. 
\end{proof}

We note that $\mathbb F_{q^2}^1$ has a non-trivial $\alpha$ if and only if $q\equiv 2 \mod 3$. We fix a nontrivial element $\alpha\neq 1$ with this property which ensures the existence of a conjugacy class of $A$ in $\U_3(q)$ which splits in $\SU_3(q)$ as follows:  $A$, $A_{1} = \diag(1, \alpha, 1).A. \diag(1, \alpha, 1)^{-1} = \begin{pmatrix} \alpha &-\alpha & -\frac{\alpha}{1+\alpha^2}\\ & \alpha &\alpha \\ &&\alpha \end{pmatrix}$, and 
$A_{2}=\diag(1, \alpha^2, 1).A. \diag(1, \alpha^2, 1)^{-1} = \begin{pmatrix} \alpha &-1 & 
-\frac{\alpha}{1+\alpha^2} \\ & \alpha & \alpha^2 \\ &&\alpha \end{pmatrix}.$

Now we note that,
\begin{lemma}\label{isolated-lu3}
The matrix $A_1$ is not isolated.
\end{lemma}
\begin{proof} Let $g=\diag(-1, 1, -1)\in \SU_3(q)$. Then $gA_1g^{-1}=\bar{A_1}^{-1}$. Therefore, $A_1$ is not an isolated matrix. 
\end{proof}

\begin{lemma}\label{isolated-lu4}
The matrix $A$ (respectively $A_2$) is conjugate to $\tr A$ (respectively $\tr A_2$) if and only if $X^3=\alpha$ has a solution in $\mathbb F_{q^2}$.  
\end{lemma}
\begin{proof} Let $B=(b_{ij})\in \SU_3(q)$ be such that $BA =\tr A B$. A simple computation leads to $B=\begin{pmatrix} && b_{13} \\ &-b_{13}\alpha^2 & b_{23}\\ b_{13} & b_{23} & b_{33} \end{pmatrix}$ with the condition $\det(B)= b_{13}^3\alpha^2 = 1$. The required condition is equivalent to $X^3=\alpha$.
\end{proof}

\begin{theorem}\label{isolated-mu5}
Let $char(\mathbb F_{q^2})\neq 2, 3$ and $A\in \SU_3(q)$. Then, the isolated matrices in $\SU_3(q)$ exist if and only if ($\mathbb F_{q^2}^1$ contains a primitive $3^{rd}$ root of unity and $\mathbb F_{q^2}$ doesn't have $9^{th}$ primitive root of unity) $q \equiv 2 \mod 3$ and $q^2 \not\equiv 1 \mod 9$. Further, there are exactly two conjugacy classes given by the matrices $A=\begin{pmatrix} \alpha & -\alpha^2 & -\frac{\alpha}{1+\alpha^2} \\ & \alpha & 1 \\ &&\alpha \end{pmatrix}$ and $A_2= \begin{pmatrix} \alpha &-1 & 
-\frac{\alpha}{1+\alpha^2} \\ & \alpha & \alpha^2 \\ &&\alpha \end{pmatrix}$ which are isolated, where $\alpha$ is a primitive $3^{rd}$ root of unity.  
\end{theorem}
\begin{proof}
From Lemmas~\ref{isolated-lu1}, \ref{isolated-lu2}, \ref{isolated-lu3}, it follows that the only possibility for isolated matrices up to conjugacy is $A$ and $A_2$ if they exist, which is equivalent to $q \equiv 2 \mod 3$. From Lemma~\ref{isolated-lu4}, it follows that $A$ and $A_2$ are not conjugate to their transpose if and only if $\mathbb F_{q^2}$ does not have a primitive $9^{th}$ root of unity which is if and only if  $q^2 \not\equiv 1 \mod 9$. Further, $A$ and $A_2$ are not real because of the eigenvalue considerations.
\end{proof}

\section{Chirality in $\G_2(q)$}\label{G2q}

In this section, we will focus on the exceptional group $\G_2(q)$ and show that it is chiral. We assume that $char(\mathbb F_q)\neq 2, 3$. We note that all automorphisms of $\G_2(q)$ are either inner or field automorphisms in this case. Also, since $\G_2(q)$ is a finite simple group, it is enough to work with automorphisms (see~\cite{Lu14}). If a group $G$ is chiral then there exists an $x$ in $G$  (said to be a \emph{witness of chirality}) such that $\varphi(x)\neq x^{-1}$ for all $\varphi\in \Aut(G)$. Further, note that only automorphisms are either inner or field automorphisms. Thus our problem reduces to finding non-real elements which are not fixed by field automorphisms. We aim to characterize such elements. First, we determine non-real elements. 

Let $k=\mathbb F_q$ be a finite field of characteristic $\neq 2,3$. Real conjugacy classes are discussed in~\cite{ST05, ST08} and are classified in terms of strongly real (which are a product of two involutions) classes. We briefly recall some of it. Let $\mathfrak C$ be an octonion (also called Cayley) algebra over $k$. It is a non-commutative, non-associative algebra of dimension $8$ over the field $k$. Since $k$ is finite, it is unique up to isomorphism, namely the split octonion algebra. The groups of type $\G_2$ over $k$ are the group of $k$-algebra automorphisms of $\mathfrak C$. Again, over a finite field, the group of type $\G_2$ is split, hence unique up to isomorphism, and is given by the algebra automorphisms of split octonion algebra. Thus, we will take $\G_2(q) = \Aut(\mathfrak C)$ for our purposes. For our work here, we need to find non-real elements in this group. We have the following:
\begin{theorem}\label{maintheorem}
Let $char(\mathbb F_q)\neq 2, 3$. The group $\G_2(q)$ has non-real conjugacy classes if and only if $q$ satisfies one of the following cases:
\begin{enumerate}
\item The field $\mathbb F_q$ has a primitive $3^{rd}$ root of unity and doesn't have a primitive $9^{th}$ root of unity, i.e., $q \equiv 1 \mod 3$ and $q \not\equiv 1 \mod 9$.
\item The field $\mathbb F_{q^2}^1$ contains a primitive $3^{rd}$ root of unity and $\mathbb F_{q^2}$ doesn't have $9^{th}$ primitive root of unity, i.e., $q \equiv 2 \mod 3$ and $q^2 \not\equiv 1 \mod 9$.
\end{enumerate}
\end{theorem}
\begin{proof}
Let $k=\mathbb F_q$ and $\mathfrak C$ be the octonion algebra over $k$. Let $t \in \Aut(\mathfrak C)$. If $t$ is unipotent, then by Lemma A.1.1 in~\cite{ST08}, it is strongly real and hence real. Thus, we may assume that we are dealing with $t$ non-unipotent. Now, if $t$ is not unipotent, then by Lemma A.1.2 in~\cite{ST08} either $t$ leaves a quaternion subalgebra invariant or fixes a quadratic \'{e}tale subalgebra $L$ of $\mathfrak C$ pointwise. If $t$ leaves a quaternion subalgebra invariant, then $t$ is strongly real and hence real. This leaves us with the case that $t$ fixes a quadratic \'{e}tale subalgebra $L$ of $\mathfrak C$ pointwise. The algebra $L \cong k\times k$ or $L$ is a quadratic field extension. In the first case $t\in \Aut(\mathfrak C/L) \cong \SL_3(k)$ and in the second case $t\in \Aut(\mathfrak C/L) \cong \SU_3(L)$. 

Let us denote the image of $t\in \Aut(\mathfrak C)$ under the map $\Aut(\mathfrak C/L) \cong \SL_3(k)$ as $A$ when $L\cong k \times k$. Similarly, Let us denote the image of $t\in  \Aut(\mathfrak C)$ under the map $\Aut(\mathfrak C/L) \cong \SU_3(L)$ when $L$ is a quadratic extension of $k$ as $A$. Note that $L\cong \mathbb F_{q^2}$ in our case. Since we are dealing with a finite field, there is a unique hermitian form to get the unitary group (up to conjugacy).  

Now, we have $t$ a non-unipotent, and further suppose characteristic polynomial $\chi_A(X) \neq m_A(X)$, the minimal polynomial. Then, from the proof of Theorem A.1.4~\cite{ST08}, $t$ is strongly real. More generally, if the characteristic polynomial of $A$ is reducible and the minimal polynomial of $A$ is not of the form $(X - \alpha)^3$, then $t$ is strongly real and hence real (see Neumann~\cite{N90}, Satz 6 and Satz 8).

Thus, we are left with the two cases (i) $\chi_A(X) =  m_A(X) = (X-\alpha)^3$, and (ii) $\chi_A(X) =  m_A(X)$ irreducible of degree three. These are the possible cases where we could have non-real 
elements, which we analyze further. We recall the following from~\cite{ST08} Proposition A.2.2 and A.3.2: 
\begin{enumerate}
\item when $L\cong k \times k$, $t$ is real in $\Aut(\mathfrak C)$ if and only if $A$ is conjugate to either $A^{-1}$ or $\tr A$ in $\SL_3(k)$. 
\item When $L$ is a quadratic field extension, $t$ is real in $\Aut(\mathfrak C)$ if and only if either $A$ or $\bar A$ is conjugate to $A^{-1}$ in $\SU_3(L)$.
\end{enumerate} 
Thus, our problem of finding non-real elements reduces to finding isolated matrices in $\SL_3(q)$ and $\SU_3(q)$. This is done in Theorem~\ref{isolated-t5} and Theorem~\ref{isolated-mu5}, which completes the proof.
\end{proof}

Given Lemma~\ref{techlemma} and Theorem~\ref{maintheorem} by noting that the non-real elements we get are not fixed by field automorphisms, we conclude the following:
\begin{corollary}\label{chiral-g2}
Let $char(\mathbb F_q)\neq 2, 3$. The group $\G_2(q)$ is chiral if and only if $q$ satisfies one of the following cases:
\begin{enumerate}
\item the field $\mathbb F_q$ has a primitive $3^{rd}$ root of unity and doesn't have a primitive $9^{th}$ root of unity, i.e., $q \equiv 1 \mod 3$ and $q \not\equiv 1 \mod 9$.
\item $\mathbb F_{q^2}^1$ contains a primitive $3^{rd}$ root of unity and $\mathbb F_{q^2}$ doesn't have $9^{th}$ primitive root of unity, i.e, $q \equiv 2 \mod 3$ and $q^2 \not\equiv 1 \mod 9$.
\end{enumerate}
\end{corollary}
\noindent Finally, we also note that, although we get two pairs of isolated matrices in the above cases, they give only one pair of non-real classes in $\G_2(q)$.


\end{document}